\documentclass{amsart}
\usepackage{enumerate}
\usepackage{calc}
\usepackage[dvips]{graphicx}
\usepackage{color}
\usepackage{pstricks}
\usepackage{amsmath, amssymb}
\usepackage{pst-node}
\newtheorem{theo}{Theorem}[section]

\newtheorem{lem}[theo]{Lemma}

\newtheorem{prop}[theo]{Proposition}

\newtheorem{defn}[theo]{Definition}

\theoremstyle{definition}
\newtheorem{rem}[theo]{Remark}
\newtheorem*{notaetoile}{Notation}

\def\a{{\a}}
\def\a{{\mathfrak a}}

\def\C{\mathbb C}

\def\E{\mathbb E}

\def\N{\mathbb N}
\def\P{\mathbb{P}}

\def\Q{\mathbb Q}
\def\R{\mathbb R}

\def\Re{\operatorname{Re}}

\def\Z{\mathbb{Z}}
\def\ch{\operatorname{ch}}
\begin{document}  

\title[Affine Lie algebras and a conditioned space-time Brownian motion]{Affine Lie algebras  and conditioned space-time Brownian motions in   affine Weyl chambers}

\author{Manon Defosseux}
\address{Laboratoire de Math\'ematiques Appliqu\'ees \`a Paris 5, Universit\'e Paris 5, 45 rue des  Saints P\`eres, 75270 Paris Cedex 06.}
\email{manon.defosseux@parisdescartes.fr}

\begin{abstract} We construct a sequence of Markov processes on the set of dominant weights of an affine Lie algebra $\mathfrak{g}$ considering  tensor product of irreducible highest weight modules of $\mathfrak{g}$ and specializations of the characters  involving the Weyl vector $\rho$. We show that it converges towards a space-time Brownian motion with a  drift, conditioned to remain in a Weyl chamber associated to the root system of $\mathfrak{g}$. This extends in particular the results of \cite{Defosseux} to any affine Lie algebras, in the case with a drift.
\end{abstract}
 
\maketitle 
\section{introduction}
 In \cite{Defosseux} we have studied a conditioned space-time Brownian motion which appears naturally in the framework of   representation theory of the
affine Lie algebra $\hat{\mathfrak{sl}_2}$:  a space-time Brownian motion $(t,B_t)_{t\ge 0}$   conditioned (in Doob's sense)  to remain in a moving boundary  domain $$D=\{(r,z)\in \R_+\times \R_+: 0< z< r\},$$
which can be seen as the Weyl chamber associated to the root system of the affine Lie algebra $\hat{\mathfrak{sl}_2}$.
The present paper deals with the case of any affine Lie algebras. Let us briefly describe  the framework of the paper.  First we need an affine Lie algebra $\mathfrak{g}$. As in the finite dimensional case, for  a dominant integral weight $\lambda$ of  $\mathfrak{g}$ one defines the character of an irreducible highest-weight representation $V(\lambda)$ of $\mathfrak{g}$ with highest weight $\lambda$, as a formal series defined for $h$ in a  Cartan subalgebra  $\mathfrak{h}$ of $\mathfrak{g}$ by  
\begin{align*} 
\mbox{ch}_\lambda(h)=\sum_{\mu}\mbox{dim} (V(\lambda)_\mu)e^{\langle \mu ,h\rangle}, 
\end{align*}
where $V(\lambda)_\mu$ is the weight space of $V(\lambda)$ corresponding to the weight $\mu$. This formal series converges for every $h$ in a subset of the Cartan subalgebra which doesn't depend on $\lambda$.  A particular choice of an element $h\in \mathfrak{h}$ in the region of convergence of the characters is called a specialization. Let us fix a dominant weight $\omega$ once for all. For  a dominant weight  $\lambda$, the following decomposition
\begin{align*}
\mbox{ch}_{\omega}\mbox{ch}_{\lambda}=\sum_{\beta\in P_+} M_{\lambda}(\beta)\mbox{ch}_{\beta},
\end{align*}
where $M_{\lambda}(\beta)$ is the multiplicity of the module with highest weight $\beta$ in the decomposition of $V(\omega)\otimes V(\lambda)$, allows to define a transition probability  $Q_\omega$ on the set of dominant weights, letting 
for $\beta$ and $\lambda$ two dominant weights of $\mathfrak{g}$, 
\begin{align}\label{qaffine}
Q_\omega(\lambda,\beta)=\frac{\mbox{ch}_\beta(h)}{\mbox{ch}_\lambda(h) \mbox{ch}_\omega(h)}M_{\lambda}(\beta),
\end{align}
 where $h$ is chosen in the region of convergence of the characters. Such a Markov chain has been recently considered by C. Lecouvey, E. Lesigne, and M. Peign\'e in \cite{LLP}.

It is a natural question to ask if there exists a  sequence $(h_n)_n$ of elements of $\mathfrak{h}$ such that the corresponding sequence of Markov chains  converges towards a continuous process and what the limit is.  One could show that there are basically three cases depending on the scaling factor. Roughly speaking the three cases are the following. When the scaling factor is $n^{-\alpha}$,  with $\alpha\in(0,1)$ (resp. $\alpha>1$), the limiting process has to do  with a Brownian motion conditioned -- in Doob's sense -- to remain in a Weyl chamber (resp.   an alcove) associated to the root system of an underlying   finite dimensional Lie algebra.  When $\alpha=1$, the limiting process has to do with a space time Brownian motion conditioned to remain in a Weyl chamber associated to the root system of the affine Lie algebra. Figure 1 below illustrates   three distinct asymptotic behaviors    in the case when the affine Lie algebra is $\hat{\mathfrak{sl}_2}$. The Weyl chamber $\mathcal C$  is the area delimited by  gray and light gray half-planes.  Essentially, when the scaling factor is  $n^{-\alpha}$ with $\alpha\in(0,1)$,  one could show that the $\Lambda_0$-component of the limiting process is $+\infty$ and that its projection  on $\R\alpha_1$ is  a Brownian motion conditioned  to remain positive.  When the scaling factor is  $n^{-\alpha}$ with $\alpha>1$, one could show that the projection of the limiting process on $\R_+\Lambda_0+\R\alpha_1$ lives in an interval  (dashed interval within Figure 1) and that its projection on $\R\alpha_1$ is a Brownian motion conditioned to remain  in an interval. When the scaling factor is $n^{-1}$, the projection of the limiting process on $\R_+\Lambda_0+\R\alpha_1$ is a space-time Brownian motion conditioned to remain in  $\mathcal C$, the  time axis being $\R_+\Lambda_0$  and the space axis being   $\R\alpha_1$.    This is this  last case which is considered in the  paper, for any affine Lie algebras. The convergence for the other values of $\alpha$ could be obtained with similar arguments as the ones developed in this paper. Nevertheless the case when $\alpha=1$ seems the most interesting case in our context as the limiting process in this case, is the only one that is really  specific to the affine framework. Thus we prefer to focus on this case. In this way we lose in generality but hope to win in clarity.      

\begin{center}
\includegraphics[scale=0.7]{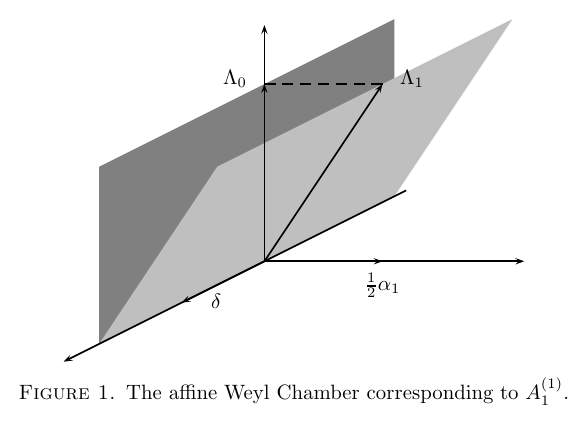}  \label{AWC}
\end{center}
The paper is organized as follows. In section \ref{A moving boundary problem} we describe   the conditioned process occuring in our setting when representations of   affine Lie algebra $\hat{\mathfrak{sl}_2}$ are considered. This is  a space-time Brownian motion with a positive drift conditioned (in Doob's sense) to remain forever in the time-dependent domain $D$. We  show by purely probabilistic arguments that  the theta functions play a crucial role in the construction of the process, which is unlighted by the algebraic point of view developed in the following sections. The vocation of this section is to give an idea of   the probabilistic aspects of mathematical objects occuring in the paper. In section \ref{Affine Lie algebras and their representations} we briefly recall the necessary background on representation theory of affine Lie algebras. We introduce  in section \ref{Markovchains} random walks on the set of integral weights of an affine Lie algebra $\mathfrak{g}$, and Markov chains on the set of its dominant integral weights, considering tensor products of irreducible highest weight representations of $\mathfrak{g}$. We show that the Weyl character formula implies that they satisfy a reflection principle. In section \ref{Scaling limit of  Random walks} we consider  a sequence of random walks obtained for particular specializations involving the Weyl vector $\rho$ of the affine Lie algebra, and prove that its scaling limit is a space-time standard Brownian motion with drift  $\rho$, living on the Cartan subalgebra of $\mathfrak{g}$. We introduce  in section \ref{Brownian} a space-time Brownian motion with drift $\rho$, conditioned to remain in an affine Weyl chamber and prove that it satisfies a reflection principle. We prove in section \ref{Scaling limit of the Markov chain} that this conditioned space-time Brownian motion is the scaling limit of a sequence of Markov processes constructed in section \ref{Markovchains}  for particular specializations involving $\rho$. 
 
\section{A moving boundary problem}\label{A moving boundary problem}
Let $\gamma\in\R$, and $(X_t,t\ge 0)=((\tau_t,B^\gamma_t),t\ge 0)$ be a space-time Brownian motion. For $(u,x)\in \R\times \R$,   $\P_{(u,x)}$ denotes a probability  under  which  $(B_t^{\gamma})_{t\ge 0}$ is a standard Brownian motion with drift $\gamma$, starting from $x$, and $\tau_t=u+t$, for all $t\ge 0$. Consider the subset $D$ of $\R^2$ defined by
$$D=\{(r,z)\in \R_+\times \R_+: 0< z< r\},$$
and consider an application $h$ defined on the closure $\bar{D}$ of $D$ by
\begin{align*} 
h(u,x):&=\P_{(u,x)}(\forall t\ge 0, 0<B_t^\gamma<\tau_t),
\end{align*}
$(u,x)\in \bar D$. When $\gamma\in(0,1),$ a classical martingale argument shows that the function $h$ is the unique bounded harmonic positive  function   for the space-time Brownian motion killed on the boundary $\partial D$,  i.e.
$$\forall (t,x)\in D,\quad (\frac{1}{2}\partial_{xx}+\gamma\partial_x+\partial_t)h(t,x)=0,$$ which satisfies  the following boundary conditions 
$$ 
   \forall t\ge 0,\quad h(t,0)=h(t,t)=0, 
$$
and the condition at infinity
$$
\underset{\tiny {\begin{array}{ll}
(t,x)\to +\infty:\\
 \frac{x}{t}\to \gamma
\end{array}}}{\lim} h(t,x)=1.
$$
Such a problem is usually referred to as a moving boundary  problem (see   for instance  \cite{Crank}  for a review of various problems specifically related to time-dependent boundaries). Actually the function $h$ can be determined using a reflection principle involving the group of  tranformations $W$ generated by linear transformations $s_k$, $k\in \Z$, defined on $\R^2$ by
$$s_k(t,x)=(t,2kt-x),\quad (t,x)\in \R^2.$$ 
Let us explain how. For $k\in \Z$, define $t_k$ as the transformation   on $\R^2$ given by 
$$t_k(t,x)=(t,2kt+x),$$ $(t,x)\in \R^2$. The group 
 $W$ is actually a semi-direct product $$ \{\mbox{Id},s_0\}\ltimes \{t_k,k\in \Z\},$$ and $\bar D$ is a fundamental domain for the action of $W$ on $\R^2$. The following proposition is immediate. 
\begin{prop} \label{Wonsemigroup} For $(u,x),(u+t,y)\in \R^2$,
\begin{align*}
\P_{s_0(u,x)}(X_t=s_0(u+t,y))&=e^{-2\gamma(y-x)}\P_{(u,x)}(X_t=(u+t,y))\\
\P_{t_k(u,x)}(X_t=t_k(u+t,y))&=e^{-2k(y-x)-2k^2t+2k\gamma t}\P_{(u,x)}(X_t=(u+t,y)),
\end{align*}
for $k\in \Z$, where $\P_{(u,x)}(X_t=(u+t,y))$  stands, by a usual abuse of notation, for the probability semi-group of $(X_t)_{t\ge 0}$. 
\end{prop} 
The probability  semi-group of the space-time Brownian motion killed on the boundary of $D$ satisfies a reflection principle. Let $\{e_1,e_2\}$ be the canonical basis of $\R^2$ and $(.,.)$ be the usual inner product on $\R^2$.  Let $T$ denote the first exit  time of $D$. The reflection principle is the following.
\begin{prop}\label{reflectioncontinuesl2}
\begin{align*}\P_{(u,x)}(X_t=(u+t,y),\, T\ge t)=e^{-\gamma x}\sum\det(r)& e^{-2k^2u-2kx+(\gamma rt_k(u,x),e_1)}
\\&\times  \P_{rt_k(u,x)}(X_t=(u+t,y)),\nonumber
\end{align*}
where the sum runs over $r\in \{\mbox{Id},s_0\}, k\in \Z$. 
\end{prop} 
\begin{proof}
As $X_T\in \bar D$, one obtains using a strong Markov property and proposition \ref{Wonsemigroup} 
 that  for $(u,x),(u+t,y)\in D$,
\begin{align*} e^{-\gamma x} \sum_{r\in \{\mbox{Id},s_0\}, k\in \Z}\det(r)& e^{-2k^2u-2kx+(\gamma rt_k(u,x),e_1)}  \P_{rt_k(u,x)}(X_t=(u+t,y), T\le t)
\end{align*}  equals $0$. 
Moreover  
\begin{align*}e^{-\gamma x}\sum_{r\in \{\mbox{Id},s_0\}, k\in \Z}\det(r)& e^{-2k^2u-2kx+(\gamma rt_k(u,x),e_1)}  \P_{rt_k(u,x)}(X_t=(u+t,y), T>t)
\end{align*}
 equals
$$\P_{(u,x)}(X_t=(u+t,y),\, T\ge t).$$
Proposition follows by summing the two identities. \end{proof}
 
One obtains for the function $h$ the following expression.
\begin{prop}\label{Danslachambresl2} For $(u,x)\in D$,
\begin{align*}
h(u,x)=2\sum_{k\in \Z}\mbox{sh}(\gamma (x+2ku))e^{-2(kx+k^2u)-\gamma x}.
\end{align*}
\end{prop}
\begin{proof} 
 Summing over $y$ such that $(t+u,y)\in D$ in proposition \ref{reflectioncontinuesl2} and letting $t$ go to infinity gives the proposition.
\end{proof}
Actually $W$ can be identified with the Weyl group associated to an affine Lie algebra $\hat{\mathfrak{sl}_2}$. Writing $X_t=\tau_t\Lambda_0+B_t^\gamma \frac{\alpha_1}{2}$, $t\ge 0$, where  $\Lambda_0$ and $\alpha_1$ are  defined below, the Doob's $h$-transform  of $(X_t)_{t\ge 0}$ is a Markov process conditioned   to remain in a Weyl chamber associated to the root system of the affine Lie algebra $\hat{\mathfrak{sl}_2}$. The following sections extend this construction to any affine Lie algebras and relate   identities from propostions \ref{reflectioncontinuesl2} and \ref{Danslachambresl2}, which are particular cases of propositions   \ref{reflectioncontinue} and \ref{Danslachambre}, to     representations theory of affine Lie algebras.

\section{Affine Lie algebras and their representations}\label{Affine Lie algebras and their representations}
In order to make the reading more pleasant, we have tried to emphasize only on definitions and properties that we need for our purpose. For more details, we refer the reader to  \cite{Kac}, which is our main reference for the whole paper.
\subsection{Affine Lie algebras}
The following definitions mainly come from chapters $1$ and $6$ of \cite{Kac}.  Let $A=(a_{i,j})_{0\le i,j\le l}$  be  a generalized Cartan matrix  of  affine type. That is all the proper principal minors of $A$ are positive and $\det A=0$.  Suppose that  rows and columns of $A$ are ordered such that   $\det \mathring{A}\ne 0$, where  $\mathring{A}= (a_{i,j})_{1\le i,j\le l}$. Let $(\mathfrak{h},\Pi,\Pi^\vee)$ be a realization of $A$ with $\Pi=\{\alpha_0,\dots,\alpha_l\}\subset \mathfrak{h}^*$ the set of simple roots, $\Pi^\vee=\{\alpha_0^\vee,\dots,\alpha_l^\vee\}\subset \mathfrak{h}$, the set of simple coroots, which satisfy the following condition 
$$\alpha_j(\alpha_i^\vee)=a_{i,j}, \, i,j\in \{0,\dots,l\}.$$
Let us consider the affine Lie algebra $\mathfrak{g}$  with generators $e_i$, $f_i$, $i=0,\dots,l$, $\mathfrak{h}$ and the following defining relations :
\begin{align*} 
&[e_i,f_i]=\delta_{ij}\alpha_i^\vee, \quad [h,e_i]=\alpha_i(h)e_i,\quad [h,f_i]=-\alpha_i(h)f_i,\\
&[h,h']=0, \textrm{ for } h,h'\in \mathfrak{h}, \\
&(\mbox{ad} e_i)^{1-a_{ij}}e_{j}=0, \quad (\mbox{ad} f_i)^{1-a_{ij}}f_{j}=0,
\end{align*} 
for all $i,j=0,\dots,l$. Let $\Delta$ (resp. $\Delta_+$) denote the set of roots (resp. positive roots) of $\mathfrak{g}$, $Q$ and $Q^\vee$ the root and the coroot lattices. We denote $a_i,$ $i,\dots,l$ the labels of the Dynkin diagram of $A$ and $a^\vee_i,$ $i=0,\dots,l$ the labels of the Dynkin diagram of $^tA$. The numbers 
$$h=\sum_{i=0}^la_i\quad and \quad h^\vee=\sum_{i=0}^la_i^\vee,$$
are called, respectively, the Coxeter number and the dual Coxeter number.  The element
$$K=\sum_{i=0}^na_i^\vee\alpha_i^\vee,$$
is called the canonical central element. The element $\delta$ defined by $$\delta=\sum_{i=0}^na_i\alpha_i,$$ 
 is the smallest positive imaginary  root.   Fix an element $d\in \mathfrak{h}$ which satisfies the following condition
$$\alpha_i(d)=0, \textrm{ for } i=1,\dots,l,\quad \alpha_0(d)=1.$$ 
The elements $\alpha_0^\vee,\dots,\alpha_l^\vee,d$, form a basis of $\mathfrak{h}$. 
We denote   $\mathfrak{h}_\R$ the linear span over $\R$ of $\alpha_0^\vee,\dots,\alpha_l^\vee,d$.    We define a nondegenerate symmetric bilinear $\C$-valued form $(.\vert .)$ on $\mathfrak{h}$ as follows
$$\left\{
  \begin{array}{ll}
(\alpha_i^\vee\vert\alpha_j^\vee)=\frac{a_j}{a_j^{\vee}}a_{ij}& i,j=0,\dots,l \\
(\alpha_i^\vee\vert d)=0& i=1,\dots,l \\
(\alpha_0^\vee\vert d)=a_0 & (d\vert d)=0. 
  \end{array}
\right.$$
 We define an element $\Lambda_0\in \mathfrak{h}^*$ by 
$$\Lambda_0(\alpha_i^\vee)=\delta_{0i}, \quad i=0,\dots,l;\quad \Lambda_0(d)=0.$$ 
The linear isomorphism 
\begin{align*}
\nu:\, \,&\mathfrak{h}\to \mathfrak{h}^*, \\
& h\mapsto (h\vert.)
\end{align*} identifies $\mathfrak{h}$ and $\mathfrak{h}^*$. We still denote   $(.\vert.)$ the induced inner product   on $\mathfrak{h}^*$. We record that
\begin{align*}
&(\delta\vert \alpha_i)=0,\quad  i=0,\dots,l,\quad (\delta\vert\delta)=0,\quad (\delta\vert\Lambda_0)=1\\
&(K\vert \alpha_i)=0, \quad i=0,\dots,l,\quad (K\vert K)=0,\quad (K\vert d)=a_0.
\end{align*} 
The form $(.\vert .)$ is $W$-invariant, for $W$  the Weyl group of the affine Lie algebra $\mathfrak{g}$, i.e.  the subgroup  of $GL(\mathfrak{h}^*)$ generated by fundamental  reflections $s_\alpha$, $\alpha\in \Pi$, defined by $$s_\alpha(\beta)=\beta-  \beta(\alpha^\vee) \alpha,\ \quad \beta\in \mathfrak{h}^*.$$  

We denote   $\mathring{\mathfrak{h}}$ (resp. $\mathring{\mathfrak{h}}_\R$)  the linear span over $\C$ (resp. $\R$) of $\alpha_1^\vee,\dots,\alpha_l^\vee$. The dual notions  $\mathring{\mathfrak{h}}^*$ and $\mathring{\mathfrak{h}}^*_\R$ are defined similarly. Then we have an orthogonal direct sum of subspaces: 
$$\mathfrak{h}=\mathring{\mathfrak{h}}_\R\oplus (\C K+\C d);\quad \mathfrak{h}^*=\mathring{\mathfrak{h}}^*_\R\oplus (\C \delta+\C \Lambda_0).$$
We set $\mathfrak{h}_\R=\mathring{\mathfrak{h}}_\R+ \R K+\R d,$ and $\mathfrak{h}_\R^*=\mathring{\mathfrak{h}}^*_\R+\R \delta+\R \Lambda_0$. 
\begin{notaetoile} For $\lambda\in \mathfrak{h}^*$ such that $\lambda=a\Lambda_0+z+b\delta$, $a,b\in \C$, $z\in \mathring{\mathfrak{h}}^*_\R$, denote   $\bar \lambda$ the projection of $\lambda$ on $\C\Lambda_0+\mathring{\mathfrak{h}}^*$ defined by $\bar\lambda=a\Lambda_0+z$, and by $\bar{\bar \lambda}$ its projection on $\mathring{\mathfrak{h}}^*$ defined by $\bar{\bar \lambda}=z$.
\end{notaetoile}
We denote   $\mathring{W}$ the subgroup  of $GL(\mathfrak{h}^*)$ generated by fundamental  reflections $s_{\alpha_i}$, $i=1,\dots,l$. Let $\Z(\mathring{W}.\theta^\vee)$ denote the lattice in $\mathring{\mathfrak{h}}_\R$ spanned over $\Z$ by the set $\mathring{W}.\theta^\vee$, where $$\theta^\vee=\sum_{i=1}^la_i^\vee\alpha_i^\vee,$$ and set $M=\nu(\Z(\mathring{W}.\theta^\vee)).$ Then $W$  is the semi-direct product $T\ltimes \mathring{W}$ (proposition 6.5 chapter 6 of \cite{Kac}) where  $T$ is the group of transformations $t_{\alpha}$, $\alpha\in M$, defined by 
$$t_\alpha(\lambda)=\lambda+\lambda(K)\alpha-((\lambda\vert \alpha)+\frac{1}{2}(\alpha\vert\alpha)\lambda(K))\delta, \quad \lambda\in \mathfrak{h}^*.$$

\subsection{Weights, highest-weight modules, characters} The following definitions and properties mainly come from chapter $9$ and $10$ of \cite{Kac}.  We denote   $P$ (resp. $P_+$) the set of integral (resp. dominant) weights defined by 
$$P=\{\lambda\in \mathfrak{h}^*:  \langle \lambda,\alpha^\vee_i \rangle\in \Z, \, i=0,\dots, l\},$$
 $$(\textrm{resp. } P_+=\{\lambda\in P:  \langle\lambda,\alpha^\vee_i \rangle\ge 0, \, i=0,\dots,l\}),$$
where $\langle.,.\rangle$ is the pairing between $\mathfrak h$ and its dual $\mathfrak h^*$.
The level of an integral weight $\lambda\in P$, is defined as the integer $(\delta\vert\lambda)$. For $k\in \N$, we denote   $P^k$ (resp. $P^k_+)$ the set of integral (resp. dominant) weights of level $k$ defined by
$$P^k=\{\lambda\in P: (\delta\vert\lambda)=k\}.$$
$$(\textrm{resp. }P^k_+=\{\lambda\in P_+: (\delta\vert\lambda)=k\}.)$$
 Recall that a $\mathfrak{g}$-module $V$ is called $\mathfrak{h}$-diagonalizable if it admits a weight space decomposition $V=\oplus_{\lambda\in \mathfrak{h}^*}V_\lambda$ by weight spaces $V_\lambda$ defined by 
$$V_\lambda=\{v\in V: \forall h\in \mathfrak{h},\, h.v=\lambda(h)v\}.$$ 
The category $\mathcal O$ is defined as the set of $\mathfrak{g}$-modules $V$ which are $\mathfrak{h}$-diagonalizable with finite dimensional weight spaces and such that there exists a finite number of elements $\lambda_1,\dots,\lambda_s\in\mathfrak{h}^*$ such that 
$$P(V)\subset \cup_{i=1}^s\{\mu\in \mathfrak{h}^*: \lambda_i-\mu\in \N \Delta_+\},$$
where $P(V)=\{\lambda\in \mathfrak{h}^*: V_\lambda\ne \{0\}\}$. One defines the formal character $\mbox{ch}(V)$ of a module $V$ from $\mathcal O$ by 
$$\mbox{ch}(V)=\sum_{\mu\in P(V)}\dim(V_\mu)e^{\mu}.$$  
For $\lambda\in P_+$ we denote   $V(\lambda)$ the irreducible module with highest weight $\lambda$. It belongs to the category $\mathcal O$. 
The Weyl character's formula (Theorem 10.4, chapter 10 of \cite{Kac}) states that 
\begin{align}\label{Weyl}
\mbox{ch}(V(\lambda))=\frac{\sum_{w\in W}\det(w)e^{w(\lambda+\rho)-\rho}}{\prod_{\alpha\in \Delta_+}(1-e^{-\alpha})^{\mbox{mult}(\alpha)}},
\end{align}
where $\mbox{mult}(\alpha)$ is the dimension of the root space $\mathfrak{g}_\alpha$ defined by 
$$\mathfrak{g}_\alpha=\{x\in \mathfrak{g}: \, \forall h\in \mathfrak{h},\, [h,x]=\alpha(h)x\},$$
for $\alpha\in \Delta$ and $\rho\in \mathfrak{h}^*$ is chosen such that $\rho(\alpha_i^\vee)=1$, for all $i\in\{0,\dots,l\}$. In particular 
\begin{align}\label{ParticularWeyl}
\prod_{\alpha\in \Delta_+}(1-e^{-\alpha})^{\mbox{mult}(\alpha)}=\sum_{w\in W}\det(w)e^{w(\rho)-\rho}.
\end{align}
Letting $e^{\mu}(h)=e^{\mu(h)}$, $h\in \mathfrak{h}$, the formal character $\mbox{ch}(V(\lambda))$ can be seen as a function   defined on its region of convergence. Actually  the  series 
$$\sum_{\mu \in P}\mbox{dim} (V(\lambda)_\mu)e^{\langle \mu ,h\rangle} $$ converges absolutely for every $h\in \mathfrak{h}$ such that $\mbox{Re}(\delta(h))>0$  (see chapter $11$ of \cite{Kac}). We denote $\mbox{ch}_\lambda(h)$ its limit.  For $\beta\in \mathfrak{h}$ such that $\Re(\beta\vert \delta)>0$, let $\mbox{ch}_\lambda(\beta)=\mbox{ch}_\lambda(\nu^{-1}(\beta))$.
\subsection{Theta functions} Connections between affine Lie algebras and theta functions are developed in chapter 13 of \cite{Kac}. We recall properties that we need for our purpose. For $\lambda\in P$ such that $(\delta\vert \lambda)=k$ one defines the classical theta function $\Theta_\lambda$ of degree $k$ by the series
$$\Theta_\lambda=e^{-\frac{(\lambda\vert \lambda)}{2k}\delta}\sum_{\alpha\in M}e^{t_\alpha(\lambda)}.$$
This series converges absolutely on $\{h\in \mathfrak{h}: \mbox{Re}(\delta(h))>0\}$ to an analytic function. 
As $$e^{\frac{(\lambda\vert\lambda)}{2k}\delta}\sum_{w\in \mathring{W}}\det(w)\Theta_{w(\lambda)}=\sum_{w\in W}\det(w)e^{w(\lambda)},$$
this last series  converges absolutely on $\{h\in \mathfrak{h}: \mbox{Re}(\delta(h))>0\}$ to an analytic function too. 

\section{Markov chains on the sets of integral or dominant weights}\label{Markovchains} Let us choose for this section a dominant weight $\omega\in P_+$  and $h\in \mathfrak{h}_\R$ such that  $ \delta(h)\in \R_+^*$.
\paragraph{\bf Random walks on $P$}  We define a probability measure $\mu_\omega$ on $P$ letting 
\begin{align}\label{weightmeasure}
\mu_\omega(\beta)=\frac{\dim(V(\omega)_\beta)}{\mbox{ch}_\omega(h) }e^{\langle \beta ,h\rangle}, \quad \beta \in P.
\end{align}
\begin{rem}\label{Fourier}
If $(X(n),n \ge 0)$ is a random walk on $P$ whose increments are distributed according to $\mu_\omega$,   keep in mind that  the function 
$$z\in  \mathring{\mathfrak{h}} _\R\mapsto\Big(\frac{\mbox{ch}_\omega(iz+h)}{\mbox{ch}_\omega(h)}\Big)^n,$$
is the Fourier transform of the projection of $X(n)$ on $\mathring{\mathfrak{h}}^* _\R$.
\end{rem}
\bigskip
\paragraph{\bf Markov chains on $P_+$}   Given two irreducible representations $V(\lambda)$ and $V(\omega)$, the tensor product of $\mathfrak{g}$-modules $V(\lambda)\otimes V(\beta)$ decomposes has a direct sum of irreducible modules. The following decomposition
$$V(\lambda)\otimes V(\omega)=\sum_{\beta\in P_+} M_{\lambda}(\beta)V(\beta),$$
where $M_{\lambda}(\beta)$ is the multiplicity of the module with highest weight $\beta$ in the decomposition of $V(\omega)\otimes V(\lambda)$, leads to the  definition a transition probability  $Q_\omega$ on $P_+$ given by
\begin{align}\label{Markovdominant}
Q_\omega(\lambda,\beta)=\frac{\mbox{ch}_\beta(h)}{\mbox{ch}_\lambda(h) \mbox{ch}_\omega(h)}M_{\lambda}(\beta),\quad \lambda,\beta \in P_+.
\end{align} 
For $n\in \N$, $\omega\in P_+$, $\beta\in P,$   denote   $m_{\omega^{\otimes n}}(\beta)$   the multiplicity  of the weight $\beta$ in $V(\omega)^{\otimes n}$. 
For $n\in \N$, $\lambda,\beta\in P_+,$  denote   $M_{\lambda,\omega^{\otimes n}}(\beta)$   the multiplicity defined by
$$V(\lambda)\otimes V(\omega)^{\otimes n}=\sum_{\beta\in P_+}M_{\lambda\otimes \omega^{\otimes n}}(\beta)V(\beta).$$ 
The Weyl character formula implies the following lemma, which is known as a consequence  of the Brauer-Klimyk rule when  $\mathfrak{g}$ is a complex semi-simple Lie algebra.
\begin{lem} \label{BrauerKlimyk}  For $n\in \N$, $\lambda,\beta\in P_+,$, one has
\begin{align*}
M_{\lambda\otimes \omega^{\otimes n}}(\beta)=\sum_{w\in W}\det(w)m_{\omega^{\otimes n}}(w(\beta+\rho)-(\lambda+\rho)),
\end{align*}
\end{lem}
\begin{proof} See  proposition 2.1 of \cite{Stembridge} and remark below. The proof is exactly the same in the framework of Kac-Moody algebras.
\end{proof} 
Let us consider the random walk $(X(n))_{n\ge 0}$ defined above and its projection $(\bar X(n))_{n\ge 0}$ on $(\R\Lambda_0+\mathring{\mathfrak{h}}^*_\R)$. Denote    $\bar P_\omega$ the transition kernel of this last random walk. The next property is immediate.
\begin{lem} Let $\beta_0,\lambda_0$ be two weights  in $ (\R\Lambda_0+\mathring{\mathfrak{h}}^*_\R)$. The transition kernel $\bar P_\omega$ satisfies 
for every $n\in \N$,
\begin{align*}
\bar P_\omega^n(\lambda_0,\beta_0) =\sum_{\beta\in P:\bar\beta=\beta_0}e^{\langle\beta-\lambda_0,h\rangle}\frac{m_{\omega^{\otimes n}}(\beta-\lambda_0)}{\mbox{ch}^n_{\omega}(h)}\\
\end{align*}
\end{lem}
Let us consider a Markov process $(\Lambda(n))_{n\ge 0}$ whose Markov kernel is given by (\ref{Markovdominant}). If $\lambda_1$ and $\lambda_2$ are two dominant weights  such that $\lambda_1= \lambda_2\,(mod\, \delta)$ then the irreducible modules $V(\lambda_1)$ and $V(\lambda_2)$ are isomorphic. Thus if we consider the random process $(\bar{\Lambda}(n),n\ge 0)$, where $\bar{\Lambda}(n)$ is the projection of  $\Lambda(n)$ on $(\R\Lambda_0+\mathring{\mathfrak{h}}^*_\R)$, then $(\bar{\Lambda}(n),n\ge 1)$ is a Markov process whose transition kernel is denoted   $\bar Q_\omega$.

\begin{prop}\label{reflectiondiscrete}  Let $\beta_0,\lambda_0$ be two dominant weights  in $( \R\Lambda_0+\mathring{\mathfrak{h}}^*_\R)$, and $n$ be a positive integer. The transition kernel $\bar Q_\omega$ satisfies  
\begin{align*}
\bar Q_\omega^n(\lambda_0,\beta_0)& =\frac{\ch_{\beta_0}(h)e^{-\langle\beta_0,h\rangle}}{\ch_{\lambda_0}(h)e^{-\langle\lambda_0,h\rangle} }  \sum_{w\in W}\det(w)e^{\langle w(\lambda_0+\rho)-(\lambda_0+\rho),h\rangle}\bar P_\omega^n(\overline{w(\lambda_0+\rho)-\rho},\beta_0)
\end{align*}
\end{prop}
\begin{proof}
Using Lemma (\ref{BrauerKlimyk}), one obtains for any dominant weight $\lambda_0,\beta_0\in (\R\Lambda_0+\mathring{\mathfrak{g}}_\R^*)$,
\begin{align*}
\bar Q_\omega^n(\lambda_0,\beta_0)&=\frac{\mbox{ch}_{\beta_0}(h)}{\mbox{ch}_{\lambda_0}(h) \mbox{ch}^n_{\omega}(h)}\sum_{\beta\in P_+:\bar\beta=\beta_0}e^{\langle\beta-\bar\beta,h\rangle}M_{\lambda,\omega^{\otimes n}}(\beta)\\
 &=\frac{\mbox{ch}_{\beta_0}(h)}{\mbox{ch}_{\lambda_0}(h) \mbox{ch}^n_{\omega}(h)}\sum_{\beta\in P:\bar\beta=\beta_0}e^{\langle\beta-\bar\beta,h\rangle}\sum_{w\in W}\det(w)m_{\omega^{\otimes n}}(w(\beta+\rho)-(\lambda+\rho)).\\ 
&=\frac{\mbox{ch}_{\beta_0}(h)e^{-\langle\beta_0,h\rangle}}{\mbox{ch}_{\lambda_0}(h)e^{-(\lambda_0,h)} }  \sum_{w\in W}\det(w)e^{\langle w(\lambda_0+\rho)-(\lambda_0+\rho),h\rangle}\bar P_\omega^n(\overline{w(\lambda_0+\rho)-\rho},\beta_0).
\end{align*} 
\end{proof}
\section{Scaling limit of  Random walks on $P$}\label{Scaling limit of  Random walks}
Let us fix $\rho=h^\vee\Lambda_0+\bar{\bar \rho}$, where $\bar{\bar \rho}$ is  half the sum of positive roots in $\mathring{\mathfrak{h}}^*$.  For $n\in \N^*$, we consider a random walk $(X^{n}(k),k\ge 0)$ starting from $0$, whose increments are distributed according to a probability measure  $\mu_{\omega}$ defined by (\ref{weightmeasure}) with $\omega\in P_+^{h^\vee}$ and $h=\frac{1}{n}\nu^{-1}( \rho)$. In particular $X^n(k)$ is an integral weight of level $h^\vee k$ for $k\in\N$. Proposition \ref{ConBrown} gives the scaling limit of the process  $(\bar{\bar{X}}^{n}(k),k\ge 0)$,

\begin{prop}\label{ConBrown} The sequence of processes $(\frac{1}{n}\bar{\bar{X}}^{n}([nt]),t\ge 0)_{n\ge 0}$ converges towards a standard Brownian motion on $\mathring{\mathfrak{h}}_\R^*$ with   drift $\bar{\bar{\rho}}$.
\end{prop} 
\begin{proof} The key ingredients for the proof are Theorems 13.8 and 13.9 of \cite{Kac}, which provide a transformation law for normalized characters. The two theorems deal with two different classes of affine Lie algebras. Let us make the proof in the framework of Theorem 13.8. The proof is  similar in the framework of  Theorem 13.9.  For the affine Lie algebras considered in   Theorem 13.8 one has  that for $n\ge 1$ and $z\in \mathring{\mathfrak{h}}^*$, 
\begin{align*}
\mbox{ch}_{\omega}&(\frac{1}{n}(\rho+z))\\
&\quad \quad =C_ne^{\frac{1}{2 n}\vert\vert\bar{\bar \rho}+z\vert\vert^2}\sum_{\Lambda\in P^{h^\vee}_+\mbox{mod}\, \C\delta }S_{\omega,\Lambda}e^{-m_\Lambda \frac{4\pi^2 n}{h^\vee}}\mbox{ch}_\Lambda(\frac{4\pi^2n}{h^\vee}\Lambda_0+2i\pi \frac{\bar{\bar\rho}+z}{h^\vee}),
\end{align*} 
where $C_n$ is a constant independent of $z$,   $m_\Lambda=\frac{\vert\vert \Lambda+\rho\vert\vert^2}{4h^\vee}-\frac{\vert\vert \rho\vert\vert^2}{2h^\vee}$ and $S_{\omega,\Lambda}$ is a coefficient independent of $z$ and $n$, for $\Lambda\in P^{h^\vee}_+$.  Notice that the sum is well-defined as for $\lambda_1=\lambda_2\,\mbox{mod}\, \C\delta$ one has 
$$e^{-m_{\lambda_1} \frac{4\pi^2 n}{h^\vee}}\mbox{ch}_{\lambda_1}(\frac{4\pi^2n}{h^\vee}\Lambda_0+2i\pi \frac{\bar{\bar\rho}+z}{h^\vee})=e^{-m_{\lambda_2} \frac{4\pi^2 n}{h^\vee}}\mbox{ch}_{\lambda_2}(\frac{4\pi^2n}{h^\vee}\Lambda_0+2i\pi \frac{\bar{\bar\rho}+z}{h^\vee}).$$

Let us prove the convergence. Let  $i\in\{1,\dots,l\}$. One has $\langle h^\vee\Lambda_0,\alpha_i^\vee\rangle=0$, which implies that $ V(h^\vee \Lambda_0)_{h^\vee \Lambda_0-\alpha_i}=\{0\}$. Consequentely,
$$\textrm{if $\beta\in P$ and $\dim(V(h^\vee\Lambda)_\beta)\ne 0,$ then  $\beta=h^\vee\Lambda_0-\sum_{k=0}^li_k\alpha_k,$}$$
where   $i_k$ is a nonnegative integer, for $k\in\{1,\dots,l\}$, and $i_0$ is a positive integer, which implies that $(\beta\vert\Lambda_0)\le -1$. Moreover, the action of   $f_k$,  for  $ k\in\{0,\dots,l\}$, on an integrable highest weight module being locally nilpotent, the number of weights $\beta$ such that $\dim(V(h^\vee\Lambda_0)_\beta)\ne 0$ and $(\beta\vert\Lambda_0)=-1$ is finite.  As the characters are defined on the set 
$$\{\lambda\in \mathfrak{h}^*: \Re(\lambda\vert \delta)>0\},$$ 
by absolutely convergent series, it implies that $$
\mbox{ch}_{h^\vee\Lambda_0} (\frac{4\pi^2n}{h^\vee}\Lambda_0+2i\pi \frac{\bar{\bar\rho}+z}{h^\vee})$$ is equal to
\begin{align*} 1+(1+\epsilon(n))e^{-\frac{4n\pi^2}{h^\vee}}\sum_{\beta: (\beta\vert \Lambda_0)=-1}\dim V(h^\vee\Lambda_0)_\beta e^{ (\beta\vert 2i\pi\frac{\bar\bar\rho+z}{h^\vee})},
\end{align*}
where $\lim_{n\to\infty}\epsilon(n)=0$. Thus
\begin{align}\label{convcar}
\lim_{n\to\infty} \Big(\mbox{ch}_{h^\vee\Lambda_0} (\frac{4\pi^2n}{h^\vee}\Lambda_0+2i\pi \frac{\bar{\bar\rho}+z}{h^\vee})\Big)^{[nt]}=1.
\end{align}

 Let $\Lambda\in P_+^{h^\vee}$ such that $(\Lambda\vert \Lambda_0)=0$. As previously,   if $\beta\in P$  and  $\dim(V(\Lambda)_\beta)\ne 0$ then $(  \beta\vert \Lambda_0)\le 0$, and  the number of weights $\beta$ such that $\dim(V(\Lambda)_\beta)\ne 0$ and $(\beta\vert \Lambda_0)= 0$ is finite. Thus,  
$$  \mbox{ch}_\Lambda(\frac{4\pi^2n}{h^\vee}\Lambda_0+2i\pi \frac{\bar{\bar\rho}+z}{h^\vee}),$$
 is bounded independently of $n$. Besides, one easily verifies that for such a $\Lambda$ one has $m_\Lambda\ge m_{h^\vee\Lambda_0}$ and that $m_\Lambda=m_{h^\vee\Lambda_0}$ implies $\Lambda=h^\vee\Lambda_0$. 
Thus 
\begin{align*}
\Big(1+\sum_{\Lambda\in P^{h^\vee}_+\setminus\{h^\vee \Lambda_0\}\mbox{mod}\, \C\delta }\frac{S_{\omega,\Lambda}}{S_{\omega,h^\vee \Lambda_0}}e^{-(m_\Lambda-m_{h^\vee\Lambda_0}) \frac{4\pi^2 n}{h^\vee}}\frac{\mbox{ch}_\Lambda(\frac{4\pi^2n}{h^\vee}\Lambda_0+2i\pi \frac{\bar{\bar\rho}+z}{h^\vee})}{\mbox{ch}_{h^\vee\Lambda_0} (\frac{4\pi^2n}{h^\vee}\Lambda_0+2i\pi \frac{\bar{\bar\rho}+z}{h^\vee})}\Big)^{[nt]}
\end{align*}
converges towards $1$ when $n$ goes to infinity.
The last convergence  and Theorem 13.8 of \cite{Kac}, recalled at the beginning of the proof, imply
$$\lim_{n\to \infty}\Big(\frac{\mbox{ch}_{\omega}(\frac{1}{n}(\rho+z))}{C_nS_{\omega,h^\vee \Lambda_0}e^{-m_{h^\vee \Lambda_0}\frac{4\pi^2n}{h^\vee}}\mbox{ch}_{h^\vee\Lambda_0} (\frac{4\pi^2n}{h^\vee}\Lambda_0+2i\pi \frac{\bar{\bar\rho}+z}{h^\vee})}\Big)^{[nt]}=e^{\frac{t}{2}\vert\vert\bar{\bar \rho}+z\vert\vert^2}.$$ 
 Finally, using convergence (\ref{convcar})  one obtains  
$$\lim_{n\to \infty}\Big(\frac{\mbox{ch}_{\omega}(\frac{1}{n}(\rho+z))}{\mbox{ch}_{\omega}(\frac{1}{n}\rho)}\Big)^{[nt]}=e^{\frac{t}{2}(\vert\vert\bar{\bar \rho}+z\vert\vert^2-\vert\vert\bar{\bar \rho}\vert\vert^2)},$$
which achieves the proof by remark (\ref{Fourier}).
\end{proof}    
\section{A conditioned space-time Brownian motion}\label{Brownian}
Denote    $\mathcal C$ the fundamental Weyl chamber defined by
$$\mathcal C=\{x\in \mathfrak{h}^* : \langle x,\alpha_i^\vee\rangle \ge0,\, i=0,\dots,l\}.$$ 
Let us consider a standard Brownian motion $(B_t)_{t\ge 0}$ on $  \mathring{\mathfrak{h}}^*_\R$.
 We consider a random process $(\tau_t\Lambda_0+B_t)_{t\ge 0}$ on $  (\R\Lambda_0+\mathring{\mathfrak{h}}^*_\R)$. For     $x\in ( \R\Lambda_0+\mathring{\mathfrak{h}}^*_\R)$, denote    $\P^0_{x}$ (resp. $\P_{x}^\rho$),  a probability under which $\tau_t=(x\vert\delta)+th^\vee,$ $\forall t\ge 0$, and  $(B_t)_{t\ge 0}$ is a standard Brownian motion  (resp.   a standard Brownian motion with drift $ \bar{\bar\rho}$) starting from $\bar {\bar x}$. Under $\P^0_{x}$ (resp. $\P^\rho_{x}$), the stochastic process $(\tau_t\Lambda_0+B_t)_{t\ge 0}$ has a transition probability semi-group $(p_t)_{t\ge 0}$ (resp. $(p^\rho_t)_{t\ge 0}$) defined by
$$p_t(x,y)=\frac{1}{(2\pi t)^{\frac{l}{2}}}e^{-\frac{1}{2t}\vert \vert y-x\vert \vert^2}1_{(y\vert \delta)=th^\vee+(x\vert \delta)},\quad x,y\in (\R\Lambda_0+\mathring{\mathfrak{h}}^* _\R).$$ 
$$(\textrm{resp. } p^\rho_t(x,y)=\frac{1}{(2\pi t)^{\frac{l}{2}}}e^{-\frac{1}{2t}\vert \vert y-\bar{\bar \rho} t-x\vert \vert^2}1_{(y\vert \delta)=th^\vee+(x\vert \delta)},\quad x,y\in (\R\Lambda_0+\mathring{\mathfrak{h}}^* _\R).)$$ Let $X_t=\tau_t\Lambda_0+B_t$, for $t\ge 0$, and consider the stopping time $T$ defined by 
$$T=\inf\{t\ge 0: X_t\notin \mathcal C\}.$$ 
The following proposition gives the probability for $(X_t)_{t\ge 0}$ to remain forever in $\mathcal C$, under $\P_x^\rho$, for $x\in \mathcal C$. 
\begin{prop}\label{Danslachambre} Let $ x\in (\R\Lambda_0+\mathring{\mathfrak{h}}^*_\R)\cap \mathcal C$. One has
\begin{align*}\P^\rho_{x}(T=+\infty) 
&=\sum_{w\in W}\det(w)e^{(  x,w(\rho)-\rho)}.
\end{align*}
\end{prop}
\begin{proof} If we consider the function $h$ defined on $(\R\Lambda_0+\mathring{\mathfrak{h}}_\R^*)$  by $$h(\lambda)=\P^\rho_{\lambda}(T=\infty), \quad \lambda\in(\R\Lambda_0+\mathring{\mathfrak{h}}_\R^*)\cap\mathcal C,$$
   usual martingal arguments state that $h$ is the unique bounded harmonic function for the killed process $(X_{t\wedge T})_{t\ge 0}$ under $\P_x^\rho$ such that 
\begin{align}\label{Boundary}
 h(\lambda)=0,\,  \textrm{ for }\lambda\in \partial \mathcal C,  \end{align} 
and
\begin{align} \label{Limite}
\lim_{t\to \infty}h(X_{t\wedge T})=1_{T=\infty}.  \end{align} 
  Let us proves that the function defined by the sum satisfies these properties. First notice that the   boundary condition (\ref{Boundary})   is satisfied. Moreover, as $x$ is in the interior of $\mathcal C$, formula (\ref{ParticularWeyl}) implies that
$$ \sum_{w\in W}\det(w)e^{(x,w(\rho)-\rho)}$$ is positive and  bounded by $1$. Choose an orthonormal basis $v_1,\dots,v_l$ of $\mathring{\mathfrak{h}}^*_\R$ and consider    for $w\in W$ a function $g_w$ defined on $\R_+^*\times \R^{l}$ by $$g_w(t,x_1,\dots,x_l)=e^{( t\Lambda_0+x,w(\rho)-\rho)},$$
where $x=x_1v_1+\dots+x_lv_l$. Letting $\Delta=\sum_{i=1}^l\partial_{x_ix_i}$, the function $g_w$ satisfies 
\begin{align}\label{edp}(\frac{1}{2}\Delta+h^\vee\partial_t+\sum_{i=1}^l(\rho,v_i)\partial_{x_i}) g_w=\frac{1}{2 }\vert\vert w(\rho)-\rho\vert\vert^2+(\rho\vert w(\rho)-\rho)=0.\end{align}
As the function $g=\sum_w\det(w)g_w$ is analytic on $\R_+^*\times \R^l$, it satisfies  (\ref{edp}) too.  Ito's Lemma implies that $(g((\tau_{t\wedge T},B_{t\wedge T}))_{t\ge 0}$ is a local martingale.  As the function $g$ is bounded by $1$ on $\{(t,x)\in \R_+^*\times \R^l: t\Lambda_0+x_1v_1+\dots+x_lv_l\in \mathcal C\}$,  $(g((\tau_{t\wedge T},B_{t\wedge T}))_{t\ge 0}$  is a martingale, i.e. $g$ is harmonic for the killed process under $\P_x^\rho$. It remains to prove that the condition (\ref{Limite}) is satisfied. For this, we notice that for any $w\in W$ distinct from the identity, $\rho-w(\rho)= \sum_{i=0}^lk_i\alpha_{i},$ where the $k_i$ are non negative integers not simultaneously equal to zero.  As almost surely
$$\lim_{t\to \infty} \frac{X_t}{t}=\rho,$$  one obtains
$$\lim_{t\to \infty} g_w(X_t)=0$$ for every $w\in W$ distinct from the identity. As the function $g$ is analytic on $\R_+^*\times \R^l$, the expected convergence   follows.
\end{proof} 

The following lemma is needed to prove a reflection principle for a Brownian motion killed on the boundary of the affine Weyl chamber.

\begin{lem}\label{Wonpt} For $x,y\in  {\mathfrak{h}}_\R^*$, $t\in \R_+$, $w\in W$, one has
 $$p^0_t(\overline{wx},\overline{wy})=e^{(w(y-x)-(y-x),h^\vee \Lambda_0)}p^0_t(\bar x,\bar y).$$
\end{lem} 
\begin{proof} Notice that $\overline{wx}=\overline{w\bar x}$. For $w\in \mathring{W}$, $\overline{wx}=w\bar x$, $p_t^0(w(\bar x),w(\bar y))=p_t^0(x,y)$ and $(wx-x\vert \Lambda_0)=(wy-t\vert \Lambda_0)=0$, which implies the identity.   For $w=t_\alpha,$ $\alpha\in M$, one has
\begin{align*}
p_t^0(\overline{wx},\overline{wy})&=p_t^0(h^\vee u\alpha+ {\bar x},h^\vee(u+t)\alpha+ {\bar y})\\
&=\frac{1}{(2\pi t)^{\frac{l}{2}}}e^{-\frac{1}{2t}\vert \vert \bar y+th^\vee\alpha-\bar x\vert \vert^2}1_{(y\vert \delta)=th^\vee+(x\vert \delta)}\\
&=p^0_t(\bar x,\bar y)e^{-\frac{1}{2t}((h^\vee)^2t^2(\alpha\vert\alpha)+2h^\vee t(\alpha\vert y-x))}\\
&=e^{(w(y-x)-(y-x),h^\vee \Lambda_0)}p^0_t(\bar x,\bar y).
\end{align*}
\end{proof}
In the following, by a classical abuse of notation, $$\P^\rho_x(X_t=y,T\ge t),\textrm{ or } \, \P^0_x(X_t=y,T\ge t),$$ $x,y\in ( \R\Lambda_0+\mathring{\mathfrak{h}}_\R^*), t\ge 0$, stands for the semi-group of the  process $(X_t)_{t\ge 0}$, with drift or not, killed on the boundary of $\mathcal C$. We first prove a reflection principle for a Brownian motion with no drift.
\begin{lem} \label{reflectionsansdrift} For $x,y\in ( \R\Lambda_0+\mathring{\mathfrak{h}}_\R^*)$ in the interior of $\mathcal C$, such that $(y\vert \delta)=(x\vert \delta)+th^\vee$, we have
\begin{align*}
\P^0_{x}(X_t=y, T> t)
&=\sum_{w\in W}\det(w)e^{(wx-x,h^\vee \Lambda_0)}p_t^0(\overline{wx},y),\\
&=\sum_{w\in W}\det(w)e^{(y-w(y),h^\vee \Lambda_0)}p_t^0(x,\overline{w(y)}).
\end{align*} 
\end{lem}
\begin{proof} Lemma \ref{Wonpt} implies in particular that we need to prove only one of the two identities. Let us prove the second one. Actually lemma \ref{Wonpt} implies that for $\alpha\in\Pi$ such that $s_\alpha(X_T)=0$ 
\begin{align*}
\E_{X_T}(1_{X_r=\overline{wy}})=e^{(wy-s_\alpha wy\vert h^\vee\Lambda_0)}\E_{X_T}(1_{X_r}=\overline{s_\alpha wy}),
\end{align*}  
which implies that 
\begin{align*}
\E_x(\sum_{w\in W}\det(w)e^{(y-wy,h^\vee\Lambda_0)}1_{T\le t,\, X_t=wy})=0.
\end{align*}
Then lemma follows from the fact that
$$\E_x(\sum_{w\in W}\det(w)e^{(y-w(y),h^\vee \Lambda_0)}1_{T> t,\, X_t=wy})=\E_x(1_{X_t=y,\,  T>t}).$$
 
\end{proof}
\begin{prop}\label{reflectioncontinue} For $x,y\in ( \R\Lambda_0+\mathring{\mathfrak{h}}_\R^*)$ in the interior of $\mathcal C$, such that $(y\vert \delta)=(x\vert \delta)+th^\vee$, we have
\begin{align*}
\P^\rho_{x}(X_t=y, T> t)&=\sum_{w\in W}\det(w)e^{(w(x)-x,\rho)}p_t^\rho(\overline{w(x)},y)\\
&=\sum_{w\in W}\det(w)e^{(y-w(y),\rho)}p_t^\rho(x,\overline{w(y)}),
\end{align*}
\end{prop} 
\begin{proof}   The result follows in a standard way from lemme \ref{reflectionsansdrift} from a Girsanov's theorem. 
\end{proof}

\section{Scaling limit of the Markov chain on $P_+$.}\label{Scaling limit of the Markov chain}
For $x\in(\R\Lambda_0+\mathring{\mathfrak{h}}^*_\R)$, proposition \ref{Danslachambre} and identity ($\ref{ParticularWeyl}$) imply in particular that the probability $P^\rho_{x}(T=+\infty)$ is positive when   $x$   is in the interior of $\mathcal C$. 
Let $(\mathcal F_t)_{t\ge 0}$ be the natural filtration of $(X_t)_{t\ge 0}$.  Let us fix $ x\in  (\R\Lambda_0+\mathring{\mathfrak{h}}^*_\R)$ in the interior of $\mathcal C$.  One considers the following conditioned process.
\begin{defn} One defines a probability $\Q^\rho_{x}$ letting 
$$\Q^\rho_{x}(A)=\E_{x}(\frac{\P^\rho_{X_t}(T=+\infty)}{\P^\rho_x(T=+\infty)}1_{T\ge t,\, A}), \textrm{ for } A\in \mathcal F_t,\, t\ge 0.$$
\end{defn}
 Under the probability $\Q^\rho_{x}$, the process $(X_t)_{t\ge 0}$ is a space-time Brownian motion with drift $ \rho$, conditioned to remain forever in the affine Weyl chamber. Let $(x_n)_{n\ge 0}$ be a sequence of elements of $P_+$ such that the sequence $(\frac{x_n}{n})_{n\ge 0}$ converges towards $x$ when $n$ goes to infinity.  For any $n\in \N^*$, we consider a Markov process $(\Lambda^{n}(k),k\ge 0)$ starting from $x_n$, with a transition probability $Q_\omega$ defined by (\ref{Markovdominant}), with $\omega\in P_+^{h^\vee}$ and $h=\frac{1}{n}\nu^{-1}( \rho)$. Notice that for $n,k\in \N$,  $\Lambda^{n}(k)$ is a dominant weight of level $kh^\vee+(x_n\vert\delta)$.    Then the following convergence holds. 

\begin{theo}
The sequence of processes $(\frac{1}{n}\bar\Lambda^{n}([nt]), t\ge 0)$ converges when $n$ goes to infinity towards the process    $(X_t,t\ge 0)$ under $\mathbb{Q}^\rho_{x}$. 
\end{theo}
\begin{proof} 
 Propositions  \ref{reflectiondiscrete} and  \ref{ConBrown} imply that the sequence of processes $(\frac{1}{n}\bar\Lambda^{n}([nt]), t\ge 0)$ converges when $n$ goes to infinity towards a Markov process with transition probability semi-group $(q_t)_{t\ge 0}$  defined by 
$$q_t(x,y)=\frac{\psi(y)}{\psi(x)}\sum_{w\in W}\det(w)e^{(w(x)-x\vert \rho)}p_t^\rho(\overline{w(x)},y), \quad x,y\in (\R\Lambda_0+\mathring{\mathfrak{h}}^*_\R),$$
where $\psi(x)=\sum_{w\in W}\det(w)e^{ (x\vert w(\rho)-\rho)}$. Propositions  \ref{Danslachambre} and \ref{reflectioncontinue} imply that 
$$q_t(x,y)=\frac{\P^\rho_y(T=+\infty)}{\P^\rho_x(T=+\infty)}\P_x(X_t=y, \,T> t), \quad x,y\in(\R\Lambda_0+\mathring{\mathfrak{h}}^*_\R),$$
which achieves the proof.
\end{proof}

\end{document}